\documentclass[12pt]{amsart}  

\usepackage{amsmath,amsthm,amssymb,amscd,mathdots,enumerate,shuffle,hyperref}
\usepackage{tikz}
\usepackage{tikz-cd}
\usetikzlibrary{decorations.pathreplacing,angles,positioning,quotes}
\allowdisplaybreaks

\newtheorem{theorem}{Theorem}[section]
\newtheorem{proposition}[theorem]{Proposition}

\newtheorem{lemma}[theorem]{Lemma}
\newtheorem{corollary}[theorem]{Corollary}

\numberwithin{equation}{section}

\theoremstyle{definition}
\newtheorem{definition}[theorem]{Definition}
\newtheorem{example}[theorem]{Example}
\newtheorem{remark}[theorem]{Remark}

\newcommand{\Q}{\mathbb{Q}}

\newcommand{\Z}{\mathbb{Z}}

\newcommand{\GL}{\operatorname{GL}}

\newcommand{\abcd}{\begin{pmatrix} a & b \\ c & d \end{pmatrix}}
\newcommand{\pmatr}[1]{\begin{pmatrix}#1\end{pmatrix}}

\newcommand{\dz}{\mathcal{D}} 
\newcommand{\de}{\mathcal{E}} 

\newcommand{\Genz}{\mathfrak{G}}

\newcommand{\zenz}{\mathfrak{z}}

\newcommand{\penz}{\mathfrak{P}}
\newcommand{\renz}{\mathfrak{R}}
\newcommand{\pzenz}{\mathfrak{p}}

\newcommand{\Benz}{\mathfrak{B}}
\newcommand{\Fenz}{\mathfrak{F}}
\newcommand{\Kenz}{\mathfrak{K}}

\DeclareRobustCommand{\gp}{\penz\!\genfrac{(}{)}{0pt}{}}
\newcommand\quotient[2]{
	\mathchoice
	{
		\text{\raise1ex\hbox{$#1$}\Big/\lower1ex\hbox{$#2$}}%
	}
	{
		#1\,/\,#2
	}
	{
		#1\,/\,#2
	}
	{
		#1\,/\,#2
	}
}

\DeclareRobustCommand{\bi}{\genfrac{(}{)}{0pt}{}}

\title{Realizations of the formal double Eisenstein space}

\author{Henrik Bachmann}
\address{Graduate School of Mathematics,  Nagoya University, Nagoya, Japan.}
\email{henrik.bachmann@math.nagoya-u.ac.jp}

\author{Ulf K\"uhn}
\address{Fachbereich Mathematik, Universit\"at Hamburg, Hamburg, Germany.}
\email{kuehn@math.uni-hamburg.de}

\author{Nils Matthes}
\address{Mathematical Institute, University of Oxford, Oxford, United Kingdom.}
\email{nils.matthes@maths.ox.ac.uk}


\begin{document}
\date{\today}
\maketitle

\begin{abstract}
We introduce the formal double Eisenstein space $\mathcal{E}_k$, which is a generalization of the formal double zeta space $\mathcal{D}_k$ of Gangl--Kaneko--Zagier, and prove analogues of the sum formula and parity result for formal double Eisenstein series. We show that $\mathbb Q$-linear maps $\mathcal{E}_k\rightarrow A$, for some $\mathbb Q$-algebra $A$, can be constructed from formal Laurent series (with coefficients in $A$) that satisfy the Fay identity. As the prototypical example, we define the Kronecker realization $\rho^{\mathfrak{K}}: \mathcal{E}_k\rightarrow \mathbb Q[[q]]$, which lifts Gangl--Kaneko--Zagier's Bernoulli realization $\rho^B: \mathcal{D}_k\rightarrow \mathbb Q$, and whose image consists of quasimodular forms for the full modular group. As an application to the theory of modular forms, we obtain a purely combinatorial proof of Ramanujan's differential equations for classical Eisenstein series.
\end{abstract}

\maketitle

\section{Introduction}

\emph{Double zeta values} are real numbers, defined for integers $k_1\geq 2$, $k_2\geq 1$, by
\[
\zeta(k_1,k_2)=\sum_{m>n>0}\frac{1}{m^{k_1}n^{k_2}} \, .
\]
The integer $k:=k_1+k_2$ is the \emph{weight} of $\zeta(k_1,k_2)$.
As was first observed by Euler, double zeta values arise naturally when multiplying the single zeta values $\zeta(k)=\sum_{n>0}n^{-k}$, for $k\geq 2$. More precisely, one has the identities
\begin{equation} \label{eqn:doubleshuffle}
\begin{aligned}
\zeta(k_1)\zeta(k_2)&=\zeta(k_1,k_2)+\zeta(k_2,k_1)+\zeta(k_1+k_2)\\
&=\sum_{j=2}^{k_1+k_2-1}\left(\binom{j-1}{k_1-1}+\binom{j-1}{k_2-1}\right)\zeta(j,k_1+k_2-j) \, ,
\end{aligned}
\end{equation}
valid for all $k_1,k_2\geq 2$, which are referred to as the \emph{double shuffle relations}.

A natural question to ask is which other relations among double zeta values can be deduced from \eqref{eqn:doubleshuffle}. In order to address this question, Gangl--Kaneko--Zagier introduced the (weight $k$) \emph{formal double zeta space} $\mathcal{D}_k$, which is $\mathbb Q$-linearly spanned by formal symbols $Z_{k_1,k_2}$, $P_{k_1,k_2}$, $Z_k$ (with $k=k_1+k_2$) subject to the double shuffle relations. Several other identities among double zeta values then hold in $\mathcal{D}_k$; for example, Euler's result (every double zeta value of odd weight can be written in terms of single zeta values and products thereof) is true in $\mathcal{D}_k$.\footnote{On the other hand, it is also proved in \cite{GKZ} that double zeta values satisfy `exotic' relations, which do not hold in $\mathcal{D}_k$, and which come from cusp forms for the full modular group.} 

Solutions to the double shuffle equations in a general $\mathbb Q$-algebra $A$ correspond to $\mathbb Q$-linear maps $\rho_k: \mathcal{D}_k\rightarrow A$, called \emph{realizations} (of $\mathcal{D}_k$). The archetypal example is the \emph{Euler realization} $\rho^E_k: \mathcal{D}_k\rightarrow \mathbb R$ that maps the generators of $\mathcal{D}_k$ to (double) zeta values in the natural way.\footnote{Some care must be taken in defining $\rho^E_k$ for $Z_{k_1,k_2}$, $P_{k_1,k_2}$ with $k_1=1$. For simplicity of the exposition, we ignore this issue in the introduction.} Another example is the \emph{Bernoulli realization} $\rho^B_k: \mathcal{D}_k\rightarrow \mathbb Q$, which sends the generator $Z_k$ to $-B_k/2k!$ and the $Z_{k_1,k_2}$ to certain linear combinations of products of Bernoulli numbers. The Bernoulli realization is of arithmetic interest as it gives an explicit construction of a rational associator in depth two, a problem raised by Drinfeld, \cite{Drinfeld:QuasiTriangular}.\footnote{More recently, Brown gave an explicit construction of a rational associator $(\tau^{(d)})_{1\leq d\leq 3}$ up to and including depth three, \cite{Brown:Depth3}. It differs from Gangl--Kaneko--Zagier's starting in weight $8$ (the smallest weight for which there exists a double zeta value that, conjecturally, cannot be written as a product of single zeta values).}

From the perspective of (quasi-)modular forms, Bernoulli numbers arise naturally as constant terms in the Fourier expansion of the classical Eisenstein series
\begin{align*}
	G_k(q) = -\frac{B_k}{2k!}+ \frac{1}{(k-1)!}\sum_{n=1}^\infty \frac{n^{k-1}q^n}{1-q^n}\,.
\end{align*}
Therefore, one may wonder whether there exists a realization with values in the algebra of quasimodular forms, which gives back the Bernoulli realization upon taking constant terms. In this paper, we address this problem by introducing the \emph{formal double Eisenstein space} $\mathcal{E}_k$, which is equipped with a canonical $\mathbb Q$-linear map $\pi_k: \mathcal{E}_k\rightarrow \mathcal{D}_k$. We construct a $\mathbb Q$-linear map $\rho^\Kenz_k: \mathcal{E}_k\rightarrow \mathbb Q[G_2,G_4,G_6]$, called the \emph{Kronecker realization}, which fits into a commutative diagram
\[
\begin{tikzcd}
\mathcal{E}_k \arrow{r}{\rho^\Kenz_k}\arrow{d}{\pi_k}& \mathbb Q[G_2,G_4,G_6]\arrow{d}{q\,\mapsto\,0}\\
\mathcal{D}_k \arrow{r}{\rho^B_k}& \mathbb Q \, ,
\end{tikzcd}
\]
where the right-hand vertical arrow maps $G_k$ to the constant term in its Fourier expansion.\footnote{The Kronecker realization should not be confused with the Eisenstein realization $\rho^{Eis}_k: \mathcal{D}_k\rightarrow \mathbb Q[[q]]$ constructed in \cite{GKZ}; indeed, the image of the latter consists of double Eisenstein series $G_{k_1,k_2}(q)$, which are in general not (quasi)modular forms. A lift of the Eisenstein realization to $\de_k$ was constructed in \cite{B2}.} The Kronecker realization is constructed from its namesake, the \emph{Kronecker function}
\[
\mathfrak{K}_q\bi{X}{Y}=-\frac{1}{2}\sum_{m=0}^\infty\frac{e^{-X-mY}q^m}{1-q^me^{-X}}+\frac{1}{2}\sum_{m=0}^\infty\frac{e^{Y+mX}q^m}{1-q^me^Y} \, ,
\]
which plays a central role in the study of elliptic functions (cf. \cite[p.73]{Weil}). For our purposes, its crucial property is the \emph{Fay identity}, \cite{BrownLevin,Mumford:Theta},
\[
\mathfrak{K}_q\bi{X_1}{Y_1}\mathfrak{K}_q\bi{X_2}{Y_2}+\mathfrak{K}_q\bi{-X_2}{Y_1-Y_2}\mathfrak{K}_q\bi{X_1+X_2}{Y_1}+\mathfrak{K}_q\bi{-X_1-X_2}{-Y_2}\mathfrak{K}_q\bi{X_1}{Y_1-Y_2}=0 \, ,
\]
familiar from the theory of theta functions and which, in this case, goes back essentially to Riemann (cf. \cite[Proposition 5]{Z}). More generally, we prove that every solution to the Fay identity can be used to construct a realization of the formal double Eisenstein space.

The formal double Eisenstein space $\mathcal{E}_k$ shares many similarities with the formal double zeta space $\mathcal{D}_k$. For example, we prove a sum formula (analogous to \cite[Theorem 1]{GKZ}) and also a parity result for $\mathcal{E}_k$. Moreover, the projection $\pi_k: \mathcal{E}_k\rightarrow \mathcal{D}_k$ admits a canonical, explicit, splitting $\sigma_k: \mathcal{D}_k\rightarrow \mathcal{E}_k$. Using these two maps, every realization $\rho: \mathcal{E}_k \rightarrow A$ induces one of $\mathcal{D}_k$ via pullback: $\rho\circ\sigma_k: \mathcal{D}_k\rightarrow A$; and vice versa with $\pi_k$ in place of $\sigma_k$. This establishes a close connection between the two spaces and its various realizations. 
As in the work of Gangl--Kaneko--Zagier, a key technical role is played by the modular group $\operatorname{SL}_2(\mathbb Z)$ and its action on polynomials. More precisely, Gangl--Kaneko--Zagier consider the standard right action of $\operatorname{SL}_2(\mathbb Z)$ on $\mathbb Q[Y_1,Y_2]$, whereas we consider a lift of this action to $\mathbb Q[X_1,X_2,Y_1,Y_2]$ which, on the $X$-variables, is the standard action conjugated by the element $S=\left(\begin{smallmatrix}0&-1\\1&0\end{smallmatrix}\right)$. We remark that, with completely different terminology, this modified action also plays a role in Ecalle's mould theory, where the Fay identity appears under the name \emph{tripartite equation}, \cite[\S 3.2]{Ecalle:Flexions}.

Finally, we mention that double zeta values are depth two versions of multiple zeta values and that the double shuffle relations naturally extend to multiple zeta values. In upcoming work \cite{BIM}, we will generalize the results of this paper by introducing the algebra of \emph{formal multiple Eisenstein series}. One can define a subalgebra of \emph{formal quasimodular forms} of this algebra for which many relations among modular forms can be proved on a purely formal level. Some results obtained in this work were announced in \cite{BBK} (with slightly different notation).
\smallskip

\noindent
\textbf{Acknowledgements:} This project started when the first-named author visited both the University of Hamburg and University of Oxford in the Summer of 2019; he would like to thank both institutions, and in particular Francis Brown, for hospitality. This project is partially supported by JSPS KAKENHI Grant 19K14499 and has received funding from the European Research Council (ERC) under the European Union's Horizon 2020 research and innovation programme (grant agreement No. 724638).

\section{Formal double Eisenstein space}
Motivated by the relations \eqref{eqn:doubleshuffle}, Gangl, Kaneko and Zagier, \cite{GKZ}, defined the \emph{formal double zeta space}, for an integer $k\geq 1$, to be the $\Q$-vector space
\begin{align*}
	\dz_k = \quotient{\big \langle Z_k , Z_{k_1,k_2}, P_{k_1,k_2} \mid k_1 + k_2 = k, k_1,k_2 \geq 1 \big \rangle_\Q}{\eqref{eq:dzrel} }\,,        
\end{align*}
which is spanned by formal symbols $Z_k$,  $Z_{k_1,k_2}$, and $P_{k_1,k_2}$, satisfying the relations 
\begin{align}
	\begin{split}
		\label{eq:dzrel} 
		P_{k_1,k_2} &= Z_{k_1, k_2} +  Z_{k_2, k_1} + Z_{k_1+k_2} \\
		&= \sum_{j=1}^{k_1+k_2-1} \left( \binom{j-1}{k_1-1} + \binom{j-1}{k_2-1} \right) Z_{j, k_1+k_2-j}\,.
	\end{split}
\end{align}
The following definition generalizes the formal double zeta space.
\begin{definition} \label{def:fdes}
For an integer $K\geq 1$, we define the \emph{formal double Eisenstein space} of weight $K$ to be 
	\begin{align*}
		\de_K = \quotient{\Big \langle G\bi{k}{d} , G\bi{k_1,k_2}{d_1,d_2}, P\bi{k_1,k_2}{d_1,d_2} \Bigm\vert \begin{array}{cc}
 k+d = k_1 + k_2+d_1+d_2 =K \\k, k_1,k_2 \geq 1,\, d, d_1,d_2\geq 0 		\end{array}\Big \rangle_\Q}{\eqref{eq:derel} }    \,,
	\end{align*}
	where we divide out the following relations 
	{\small
	\begin{align}\begin{split} \label{eq:derel}
			P\bi{k_1,k_2}{d_1,d_2}  &=  G\bi{k_1,k_2}{d_1,d_2} +G\bi{k_2,k_1}{d_2,d_1} +G\bi{k_1+k_2}{d_1+d_2} \\
			&=  \sum_{\substack{l_1+l_2=k_1+k_2\\ e_1+e_2=d_1+d_2\\l_1,l_2\geq 1, e_1,e_2\geq 0}}\!\!\! \left(\binom{l_1-1}{k_1-1}\binom{d_1}{e_1}(-1)^{d_1-e_1} +   \binom{l_1-1}{k_2-1}\binom{d_2}{e_1} (-1)^{d_2-e_1}  \right) G\bi{l_1,l_2}{e_1,e_2} \\&+\frac{d_1! d_2!}{(d_1+d_2+1)!}\binom{k_1+k_2-2}{k_1-1}G\bi{k_1+k_2-1}{d_1+d_2+1}\,.
		\end{split}
	\end{align}}
\end{definition}

\begin{remark} The origin of the formal double Eisenstein space lies in the study of the algebraic structure of certain $q$-analogues of multiple zeta values, introduced by the first author in \cite{B1}. A survey of these $q$-analogues and motivation for the definition of the formal double Eisenstein space can be found in \cite{B2}.
\end{remark}

In weight $K=1$, there is exactly one element $G\bi{1}{0}$ and no relations. For $K=2$, one has $P\bi{1,1}{0,0} = 2 G\bi{1,1}{0,0} + G\bi{2}{0} =  2 G\bi{1,1}{0,0} + G\bi{1}{1}$, which gives $G\bi{2}{0}=G\bi{1}{1}$, and $\dim_\Q \de_2 = 2$. This is slightly different to the formal double zeta space, where $Z_2 = 0$ and $\dim_\Q \dz_2 = 1$. Calculating all relations in small weights, we obtain the following table
	\begin{center}
	\begin{tabular}{|c|c|c|c|c|c|c|c|c|c|c|c|c|}
		\hline
		 $k$          & 1 & 2 & 3 & 4 & 5 & 6 & 7 & 8 & 9 & 10 & 11 & 12 \\ \hline
		$\dim_\Q \dz_k$     & 1 & 1 & 2 & 2 & 3 & 3 & 4 & 4 & 5 & 5  & 6 & 6 \\ \hline
		$\dim_\Q \de_k$     & 1 & 2 & 5 & 8 & 15& 22 & 35 & 48 & 69 & 90 & 121 & 152 \\ \hline
	\end{tabular}
\end{center}

\begin{remark}
It is shown in \cite{GKZ} that $\dim_\Q \dz_k = \lfloor \frac{k+1}{2} \rfloor$, i.e., that the relations in \eqref{eq:dzrel} are linearly independent except for the symmetry coming from interchanging $k_1$ and $k_2$. We verified numerically that this also seems to be the case for the relations \eqref{eq:derel}, i.e., that these relations are all linearly independent except for the symmetry of interchanging $(k_1,d_1)$ and $(k_2,d_2)$. Counting generators and possible independent relations gives an explicit conjecture for the dimension of $\de_k$, which we will not discuss in this note. 
\end{remark}

\begin{definition}
Let $A$ be a $\Q$-algebra. An \emph{$A$-valued realization} of $\de_k$ is a $\Q$-linear map $\de_k\rightarrow A$. Likewise, an \emph{$A$-valued realization} of $\dz_k$ is a $\Q$-linear map $\dz_k\rightarrow A$.
\end{definition}
We will frequently omit explicit mention of $A$, if it is clear from context. 

In Section \ref{sec:realization}, we will construct explicit realizations of $\de_k$  for $A = \Q[[q]]$. Before doing so, we will compare the spaces $\de_k$ and $\dz_k$ in more detail. To work with the formal double Eisenstein space, it is convenient to consider all weights at the same time and to use generating series, which we define as follows
{\small 
	\begin{align*}
		\Genz_1\!\bi{X}{Y} := &\sum_{\substack{k\geq 1\\d \geq 0}} G\bi{k}{d} X_1^{k-1}   \frac{Y_1^{d}}{d!}  \,, \qquad 
		\Genz_2\!\bi{X_1,X_2}{Y_1,Y_2} := \sum_{\substack{k_1,k_2\geq 1\\d_1,d_2 \geq 0}} G\bi{k_1,k_2}{d_1,d_2} X_1^{k_1-1}  X_2^{k_2-1} \frac{Y_1^{d_1}}{d_1!}  \frac{Y_2^{d_2}}{d_2!}\,, \\
		\gp{X_1,X_2}{Y_1,Y_2} := &\sum_{\substack{k_1,k_2\geq 1\\d_1,d_2 \geq 0}} P\bi{k_1,k_2}{d_1,d_2} X_1^{k_1-1}  X_2^{k_2-1} \frac{Y_1^{d_1}}{d_1!}  \frac{Y_2^{d_2}}{d_2!}\,.	
\end{align*}} 
The defining equations \eqref{eq:derel} for the space $\de_k$ can then be written, for all $k \geq 1$, as
\begin{align}\label{eq:qdsh}
	\begin{split}
		\gp{X_1,X_2}{Y_1,Y_2}&=  \Genz_2\!\bi{X_1,X_2}{Y_1,Y_2} + \Genz_2\!\bi{X_2,X_1}{Y_2,Y_1} + \frac{\Genz_1\!\bi{X_1}{Y_1+Y_2} -\Genz_1\!\bi{X_2}{Y_1+Y_2}}{X_1-X_2} \\
		&= \Genz_2\!\bi{X_1+X_2, X_2}{Y_1, Y_2-Y_1}+\Genz_2\!\bi{X_1+X_2,X_1}{Y_2, Y_1-Y_2}  + \frac{\Genz_1\!\bi{X_1+X_2}{Y_1}-\Genz_1\!\bi{X_1+X_2}{Y_2}}{Y_1-Y_2} \,.
	\end{split}
\end{align}
Similarly, the defining equations \eqref{eq:dzrel} for the formal double zeta space can be expressed by using the generating series 
\begin{align*}
    \zenz_1(X) &:= \sum_{k\geq 1} Z_k X^{k-1}\,,\quad  \zenz_2(X_1,X_2) := \sum_{k_1,k_2\geq 1} Z_{k_1,k_2} X_1^{k_1-1} X_2^{k_2-1}\,,\\
    \pzenz(X_1,X_2) &:= \sum_{k_1,k_2\geq 1} P_{k_1,k_2} X_1^{k_1-1} X_2^{k_2-1}\,.
\end{align*}
Then \eqref{eq:dzrel} for all $k\geq 1$ is equivalent to
\begin{align}\begin{split}\label{eq:mouldfulldshdep2}
\pzenz(X_1,X_2) &= \zenz_2(X_1,X_2) + \zenz_2(X_2,X_1) + \frac{\zenz_1(X_1)- \zenz_1(X_2)}{X_1- X_2}\\
&= \zenz_2(X_1+X_2,X_2) + \zenz_2(X_1+X_2,X_1).
\end{split}
\end{align}

\begin{proposition}\label{prop:mappi} For all $k\geq 1$, there is a well-defined $\Q$-linear map $\pi_k: \de_k \rightarrow \dz_k$ given by
\begin{align*}
 \pi_k:   G\bi{k}{d} &\longmapsto \delta_{d,0} Z_k + \delta_{k,1} d! Z_{d+1}\,,\\
  \pi_k:  G\bi{k_1,k_2}{d_1,d_2} &\longmapsto \delta_{(d_1,d_2),(0,0)}\, Z_{k_1,k_2} + \delta_{(k_1,k_2),(1,1)}\, \sum_{\substack{a+b=d_1+d_2+2\\a,b \geq 1}} \frac{d_1! (a-1)! }{(a-1-d_2)!} Z_{a,b} \,.
\end{align*}
\end{proposition}
\begin{proof} We need to show that $\pi_k$ maps every relation defining $\de_k$ to a relation defining $\dz_k$. This is straightforward by writing the map $\pi_k$ in terms of generating series. More precisely, denote by $\pi$ the map defined on a power series by applying $\pi_k$ on the coefficients of the degree $k-2$ part. Then
\begin{align*}
    \pi\left( \Genz\bi{X}{Y} \right) &= \zenz(X) + \zenz(Y)\,,\\
    \pi\left( \Genz\bi{X_1,X_2}{Y_1,Y_2} \right) &= \zenz(X_1,X_2) + \zenz(Y_1+Y_2,Y_1) \,.
\end{align*}
Applying $\pi$ to the relations \eqref{eq:derel}, we obtain a relation for $\dz_k$ by using \eqref{eq:dzrel}.
\end{proof}

In the other direction, we have the following proposition.

\begin{proposition} \label{prop:dktoek}
	For all  $k\geq 1$, there is a well-defined $\Q$-linear map $\sigma_k: \dz_k \rightarrow \de_k$ given by
	\begin{align*}
	\sigma_k:	Z_k &\longmapsto G\bi{k}{0} - \delta_{k,2} G\bi{2}{0} \,,\\ 
	\sigma_k:  Z_{k_1,k_2} &\longmapsto G\bi{k_1,k_2}{0,0} + \frac{1}{2} \left(  \delta_{k_2,1}G\bi{k_1}{1} -  \delta_{k_1,1} G\bi{k_2}{1} +   \delta_{k_1,2} G\bi{k_2+1}{1} \right)  \,,\\
	\sigma_k: 	P_{k_1,k_2} &\longmapsto 	P\bi{k_1,k_2}{0,0}+ \frac{1}{2} \left(    \delta_{k_1,2} G\bi{k_2+1}{1} + \delta_{k_2,2} G\bi{k_1+1}{1} \right) - \delta_{k_1 \cdot k_2,1} G\bi{2}{0} \,.
	\end{align*}
\end{proposition}
\begin{proof}
As before, the content of this proposition is that $\sigma_k$ is well-defined. This can be done by applying this map to the defining equation \eqref{eq:dzrel} of $\dz_k$ and then observe that, after rearranging some terms and using $G\bi{2}{0}=G\bi{1}{1}$, the result is exactly  \eqref{eq:derel} with $d_1=d_2=0$.
\end{proof}
By Proposition \ref{prop:dktoek}, every relation in the formal double zeta space gives rise to a relation in the formal double Eisenstein space, and every realization $\rho$ of $\de_k$ gives a realization $\rho \circ \sigma_k$ of $\dz_k$. For $k\geq 3$, the map $\pi_k \circ \sigma_k$ is the identity on $\dz_k$ and we have the split exact sequence
\begin{center}
    \begin{tikzcd}
    0\arrow{r} & \ker(\pi_k)  \arrow{r} & \de_k \arrow{r}{\pi_k}& \dz_k \arrow{r} \arrow[bend left=33]{l}{\sigma_k}  & 0
\end{tikzcd}\,.
\end{center}

\begin{proposition}\label{prop:partialk} For all  $k\geq 1$, the following gives a $\Q$-linear map $\partial_k: \de_k \rightarrow \de_{k+2}$
\begin{align*}
    \partial_k: G\bi{k}{d} &\longmapsto k G\bi{k+1}{d+1}\,,\\
    \partial_k: G\bi{k_1,k_2}{d_1,d_2} &\longmapsto k_1 G\bi{k_1+1,k_2}{d_1+1,d_2} + k_2 G\bi{k_1,k_2+1}{d_1,d_2+1}\,.  
\end{align*}
\end{proposition}
\begin{proof}
This can be checked on the level of coefficients by applying $\partial_k$ to \eqref{eq:derel}. Alternatively, one may again use generating series as follows. Denote by $\partial$ the map on generating series which is given on the coefficients of the degree $k-2$ part by $\partial_k$. Then,
\begin{align*}
    \partial \Genz\bi{X}{Y} = \frac{d}{dX}\frac{d}{dY} \Genz\bi{X}{Y}\,,\qquad  \partial \Genz\bi{X_1,X_2}{Y_1,Y_2} = \left(\frac{d}{dX_1}\frac{d}{dY_1} + \frac{d}{dX_2}\frac{d}{dY_2}   \right) \Genz\bi{X_1,X_2}{Y_1,Y_2}\,.
\end{align*}
The statement now follows from verifying that
\begin{align*}
\left(\frac{d}{dX_1}\frac{d}{dY_1} + \frac{d}{dX_2}\frac{d}{dY_2}   \right) = \left(\frac{d}{d(X_1+X_2)}\frac{d}{dY_2} + \frac{d}{dX_1}\frac{d}{d(Y_1-Y_2)}   \right) \, .
\end{align*}
Therefore, applying $\partial$ to \eqref{eq:qdsh} is equivalent to applying the above differential operator. One concludes from this that $\partial_k$ maps every relation for $\de_k$ to a relation for $\de_{k+2}$.
\end{proof}
We will see in Proposition \ref{prop:derivcommute} that, for the Kronecker realization, the operator $\partial_k$ corresponds to the $q$-derivative differential operator $q\frac{d}{dq}$, which appears in the theory of quasimodular forms.
\section{Constructing realizations}\label{sec:realization}

Finding a realization of $\de_k$ in some $\mathbb Q$-algebra $A$ for all $k\geq 1$ is equivalent to finding power series with coefficients in $A$ satisfying \eqref{eq:qdsh}. It will be convenient to describe the latter using an action of the group ring $\Z[\GL_2(\Z)]$ on formal power series $A[[X_1,X_2,Y_1,Y_2]]$, which is defined as follows. Given $\gamma = \abcd \in \GL_2(\Z)$ and $R\in A[[X_1,X_2,Y_1,Y_2]]$, define 
\begin{align}\label{eq:defmatthesaction}
R\!\mid\!\gamma\bi{X_1,X_2}{Y_1,Y_2} := R\bi{a X_1 + b X_2, c X_1 + d X_2}{\det(\gamma)(d Y_1 - c Y_2), \det(\gamma) (-bY_1 + a Y_2)}\,
\end{align}
and extend this action $\mathbb Z$-linearly to $\Z[\GL_2(\Z)]$. In the following, we will use the matrices
\begin{align}\begin{split}\label{eq:matrixdef}
\sigma&= \pmatr{-1 & 0\\0 & -1}\,,\qquad
\epsilon = \pmatr{0 & 1 \\ 1 & 0}\,,\qquad \delta=\pmatr{-1 & 0\\0 & 1} \,,    \\
T &= \pmatr{1 & 1\\ 0 & 1}\,,\qquad S=\pmatr{0 & -1\\1 & 0}\,,\qquad U = \pmatr{1 & -1\\1 & 0}\,.
\end{split}
\end{align}
Equation \eqref{eq:qdsh} can then be written as
\begin{align*}
		\penz &=  \Genz_2 {\mid\! (1+\epsilon)} + \renz_{\Genz_1}^\ast =    \Genz_2 {\mid T(1+\epsilon)} + \renz_{\Genz_1}^\shuffle \,,
\end{align*}
\newcommand{\ps}{\mathfrak{T}}
where we define, for a power series $\ps\bi{X}{Y}\in A[[X,Y]]$, the following elements in $A[[X_1,X_2,Y_1,Y_2]]$
\begin{align}\label{eq:defR}
	 \renz_{\ps}^\ast\bi{X_1,X_2}{Y_1,Y_2} := \frac{\ps\!\bi{X_1}{Y_1+Y_2} -\ps\!\bi{X_2}{Y_1+Y_2}}{X_1-X_2}\,,\quad  \renz_{\ps}^\shuffle\bi{X_1,X_2}{Y_1,Y_2} :=\frac{\ps\!\bi{X_1+X_2}{Y_1}-\ps\!\bi{X_1+X_2}{Y_2}}{Y_1-Y_2}\,.
\end{align}
In order to construct realizations of $\de_k$, we need to find power series $\penz_\Benz,\Benz_1,\Benz_2 \in A[[X_1,X_2,Y_1,Y_2]]$ which satisfy 
\begin{align}\label{eq:actiongenrelation}
		\penz_\Benz &=  \Benz_2 {\mid\! (1+\epsilon)} + \renz_{\Benz_1}^\ast =    \Benz_2  {\mid T(1+\epsilon)} + \renz_{\Benz_1}^\shuffle \,,
\end{align}
This can be done similarly as in \cite{GKZ} for the formal double zeta space. More precisely, we consider all weights $k$ at the same time and consider the following space of Laurent series
\begin{equation*}
\mathcal{W}^+\!(A) := \left\{ R \in A((X_1,X_2,Y_1,Y_2)) \,\,\big\vert\,\, R\!\mid\!(1+U+U^2)=R\!\mid\!(1+S)=R\!\mid\!(1-\epsilon)=0   \right\},\footnote{We remark that the action of $\operatorname{SL}_2(\mathbb Z)$ on $A[[X_1,X_2,Y_1,Y_2]]$ does not extend to formal Laurent series. Nevertheless, the equations defining the space $\mathcal{W}^+\!(A)$ can be viewed as identities in the field of fractions $\operatorname{Frac}(A((X_1,X_2,Y_1,Y_2)))$, and hence $\mathcal{W}^+\!(A)$ is well-defined.}
\end{equation*}
where $A((X_1,X_2,Y_1,Y_2))$ denotes the $A$-algebra of all formal Laurent series of the form
\[
R\bi{X_1,X_2}{Y_1,Y_2}=\sum_{\substack{m_1,m_2>\!\!\!>-\infty \\ n_1,n_2>\!\!\!>-\infty}}a\bi{m_1,m_2}{n_1,n_2}X_1^{m_1}X_2^{m_2}Y_1^{n_1}Y_2^{n_2}, \qquad a\bi{m_1,m_2}{n_1,n_2} \in A \,.
\]
The space $\mathcal{W}^+\!(A)$ may be seen as the collection of all odd/symmetric `bi'-period polynomials. Given one of its elements, we next construct a realization of $\de_k$, and then discuss how elements in $\mathcal{W}^+\!(A)$ may be constructed from certain Laurent series satisfying the \emph{Fay-identity}. We begin with the following analogue of \cite[Proposition 5]{GKZ}, which associates to an element of $\mathcal{W}^+\!(A)$ a solution of \eqref{eq:actiongenrelation} with $\Benz_1=0$.
\begin{lemma} \label{lem:prop5analogue}
	For $\widetilde{\penz} \in \mathcal{W}^+\!(A)$ the element $\widetilde{\Benz}_2 := \frac{1}{3} \widetilde{\penz} |(1+T^{-1})$ satisfies
	\begin{align*}
	\widetilde{\penz}  &= \widetilde{\Benz}_2|(1+\epsilon)  =  \widetilde{\Benz}_2 \mid T(1+\epsilon) \,.
	\end{align*}
\end{lemma}
\begin{proof} The proof is similar to the one of \cite[Proposition 5 (ii)]{GKZ}. First, by \eqref{eq:matrixdef}, we have $T=US^{-1}$, and by setting $T':=U^2 S^{-1}$ we get 
\begin{align*}
    1-T-T'=(1+S)S^{-1} - (1+U+U^2)S^{-1}
\end{align*}
and therefore $\widetilde{\penz}\mid (1-T-T')=0$. Further, since $\widetilde{\penz}\mid S = -\widetilde{\penz}$, $\widetilde{\penz}\mid \epsilon = \widetilde{\penz}$ and $S=\delta \epsilon$, $\sigma=\epsilon \delta \epsilon \delta$, we get $\widetilde{\penz}\mid \delta = -\widetilde{\penz}$ as well as $\widetilde{\penz}\mid \sigma = \widetilde{\penz}$. This together with $\sigma T' T^{-1} = \delta T^{-1}\epsilon$ and $T' = \epsilon T \epsilon$ gives
\begin{align*}
    \widetilde{\penz}\mid\! (1+T^{-1})(1+\epsilon) &= 2   \widetilde{\penz} +   \widetilde{\penz}\mid(T^{-1}-\delta T^{-1}\epsilon) = 3  \widetilde{\penz} +   \widetilde{\penz}\mid(1-T-T')T^{-1} = 3  \widetilde{\penz}\,,\\
    \widetilde{\penz}\mid\! (1+T^{-1})T(1+\epsilon) &=  \widetilde{\penz}\mid\! (1+T)(1+\epsilon) = 2 \widetilde{\penz}   + \widetilde{\penz}\mid (T+T') = 3  \widetilde{\penz}\,.
\end{align*}
\end{proof}
Notice that not all elements constructed in Lemma \ref{lem:prop5analogue} solve the original problem to find a solution of \eqref{eq:actiongenrelation}, since $\widetilde{\penz}$ might have poles. Instead, to construct solutions of \eqref{eq:actiongenrelation} which are nonzero in depth one and which are power series without poles, we consider elements $\widetilde{\penz} \in \mathcal{W}^+\!(A)$ of the following form
\begin{align}\label{eq:defppol}
\widetilde{\penz}_\Benz\bi{X_1,X_2}{Y_1,Y_2} &= \underbrace{-\frac{1}{2}\left( \left(\frac{1}{X_2} + \frac{1}{Y_2} \right) \Benz_1\bi{X_1}{Y_1} +  \left(\frac{1}{X_1} + \frac{1}{Y_1} \right) \Benz_1\bi{X_2}{Y_2}  \right)}_{\widetilde{\penz}_\Benz^{\rm{pol}}\bi{X_1,X_2}{Y_1,Y_2} := } + \penz_\Benz\bi{X_1,X_2}{Y_1,Y_2}\,,
\end{align}
for some  $\Benz_1\bi{X}{Y} \in A[[X,Y]]$ and $\penz_\Benz\bi{X_1,X_2}{Y_1,Y_2} \in A[[X_1,X_2,Y_1,Y_2]]$. \\

\begin{lemma}\label{lem:beta} Let $\Benz_1\bi{X_1}{Y_1} \in A[[X_1,Y_1]]$ be with $\Benz_1\bi{-X_1}{-Y_1}=-\Benz_1\bi{X_1}{Y_1}$ and 
	\begin{align}\label{eq:defbeta}
	\beta := \frac{1}{4} \renz_{\Benz_1}^\ast {\mid\! (5 -3 U + U \epsilon)} +  \frac{1}{4} \renz_{\Benz_1}^\shuffle {\mid\! (T^{-1} ( 5 - 3 \epsilon + U))}\,,
	\end{align}
	where $\renz_{\Benz_1}^\ast$ and $\renz_{\Benz_1}^\ast$ are given by \eqref{eq:defR}. Then we have
	\begin{align*}
	\beta {\mid\! (1+\epsilon)} &= 3  \renz_{\Benz_1}^\ast+ \widetilde{\penz}_\Benz^{\rm{pol}} {\mid\! (1 - T^{-1}- T^{-1} \epsilon)}\,,\\
	\beta{\mid T(1+\epsilon)} &= 3  \renz_{\Benz_1}^\shuffle + \widetilde{\penz}_\Benz^{\rm{pol}} {\mid\! (1 - T- T \epsilon)}\,,
	\end{align*}
	where $\widetilde{\penz}_\Benz^{\rm{pol}}$ is given by \eqref{eq:defppol}.
\end{lemma}
\begin{proof}
	One checks by direct computation that 
	{\small	\begin{align*}
	3  \renz_{\Benz_1}^\ast+ \widetilde{\penz}_\Benz^{\rm{pol}}\mid\! (1 - T^{-1}- T^{-1} \epsilon) - \beta\mid\! (1+\epsilon)  =&\frac{1}{4(Y_1+Y_2)}\left( \Benz_1\bi{X_1-X_2}{Y_1}+\Benz_1\bi{-X_1+X_2}{-Y_1} \right)\\
	&\frac{1}{4(Y_1+Y_2)} \left(\Benz_1\bi{X_1-X_2}{-Y_2}+\Benz_1\bi{-X_1+X_2}{Y_2} \right).
	\end{align*}}Using the equality $\Benz_1\bi{-X_1}{-Y_2}=-\Benz_1\bi{X_1}{Y_2}$, we see that the right-hand side vanishes. A similar calculation yields the second equation.
\end{proof}
The main result of this section is the following theorem.
\begin{theorem} \label{thm:solfromwp}Let $\widetilde{\penz}_\Benz \in \mathcal{W}^+\!(A)$ be of the form 
\begin{align*}
\widetilde{\penz}_\Benz\bi{X_1,X_2}{Y_1,Y_2} &= -\frac{1}{2}\left( \left(\frac{1}{X_2} + \frac{1}{Y_2} \right) \Benz_1\bi{X_1}{Y_1} +  \left(\frac{1}{X_1} + \frac{1}{Y_1} \right) \Benz_1\bi{X_2}{Y_2}  \right) + \penz_\Benz\bi{X_1,X_2}{Y_1,Y_2}\,,
\end{align*}
for some  $\Benz_1\bi{X_1}{Y_1},\penz_\Benz\bi{X_1,X_2}{Y_1,Y_2} \in A[[X_1,X_2,Y_1,Y_2]]$ with $\Benz_1\bi{-X_1}{-Y_1}=-\Benz_1\bi{X_1}{Y_1}$. Then 
	\[  \Benz_2 :=\frac{1}{3} \penz_\Benz {\mid\! (1+T^{-1})}-\frac{1}{12} \renz_{\Benz_1}^\ast {\mid\! (5 -3 U + U \epsilon)} -\frac{1}{12}\renz_{\Benz_1}^\shuffle {\mid\! (T^{-1} ( 5 - 3 \epsilon + U))}\] solves \eqref{eq:actiongenrelation}, i.e.,
\begin{align*}
		\penz_\Benz &=  \Benz_2 {\mid\! (1+\epsilon)} + \renz_{\Benz_1}^\ast \\
		&=    \Benz_2  {\mid T(1+\epsilon)} + \renz_{\Benz_1}^\shuffle \,.
\end{align*}
We obtain in particular a realization of $\de_k$, for all $k\geq 1$, with values in $A$.
\end{theorem}
\begin{proof} We only verify the first equation; the verification of the second one is similar. First, notice that we can write $\Benz_2 = \frac{1}{3} \penz_\Benz {\mid\! (1+T^{-1})} - \frac{1}{3}\beta$, where $\beta$ is given as in \eqref{eq:defbeta}. Writing $\penz_\Benz = \widetilde{\penz}_\Benz - \widetilde{\penz}_\Benz^{\rm{pol}}$, using Lemma \ref{lem:prop5analogue}, Lemma \ref{lem:beta}, and $\widetilde{\penz}_\Benz^{\rm{pol}} | \epsilon = \widetilde{\penz}_\Benz^{\rm{pol}}$, we obtain
\begin{align*}
    \Benz_2 {\mid\! (1+\epsilon)} = \,&\frac{1}{3} \left( \widetilde{\penz}_\Benz - \widetilde{\penz}_\Benz^{\rm{pol}}\right) {\mid\! (1+T^{-1}) (1+\epsilon)}  - \frac{1}{3}\beta{\mid\! (1+\epsilon)} \\
    \overset{\text{Lem.} \ref{lem:prop5analogue}}{=}\,\,  &\widetilde{\penz}_\Benz -\frac{1}{3} \widetilde{\penz}_\Benz^{\rm{pol}} {\mid\! (1+T^{-1}+\epsilon+T^{-1}\epsilon)}  - \frac{1}{3}\beta{\mid\! (1+\epsilon)}\\
    \overset{\text{Lem.} \ref{lem:beta}}{=}\,\, &\widetilde{\penz}_\Benz -\frac{1}{3} \widetilde{\penz}_\Benz^{\rm{pol}} {\mid\! (1+T^{-1}+\epsilon+T^{-1}\epsilon)} - \renz_{\Benz_1}^\ast -\frac{1}{3}\widetilde{\penz}_\Benz^{\rm{pol}} {\mid\! (1 - T^{-1}- T^{-1} \epsilon)} \\
    = \,&\widetilde{\penz}_\Benz - \widetilde{\penz}_\Benz^{\rm{pol}}  - \renz_{\Benz_1}^\ast = \penz_\Benz - \renz_{\Benz_1}^\ast\,,
\end{align*}
as was to be shown.
\end{proof}
We now single out a special case of Theorem \ref{thm:solfromwp}. We shall say that an element $\Fenz\bi{X_1}{Y_1} \in A((X_1,X_2,Y_1,Y_2))$ satisfies the \emph{Fay identity}, if 
\begin{align}\label{eq:fay}
\Fenz\bi{X_1}{Y_1}\Fenz\bi{X_2}{Y_2} + \Fenz\bi{X_1-X_2}{-Y_2}\Fenz\bi{X_1}{Y_1+Y_2}+\Fenz\bi{-X_2}{-Y_1-Y_2}\Fenz\bi{X_1-X_2}{Y_1} = 0\,.     
\end{align}
Furthermore, if $\Fenz$ is odd, $\Fenz\bi{-X}{-Y} =-\Fenz\bi{X}{Y}$, then  
\begin{align*}
    \widetilde{\penz}_\Fenz\bi{X_1,X_2}{Y_1,Y_2} := \Fenz\bi{X_1}{Y_1}\Fenz\bi{X_2}{Y_2}
\end{align*}
is an element of $\mathcal{W}^+\!(A)$, since the Fay identity \eqref{eq:fay} is equivalent to 
\begin{align*}
    \widetilde{\penz}_\Fenz \mid (1+U+U^2)=0\,.
\end{align*}
Also, since $\Fenz$ is odd, we get $\widetilde{\penz}_\Fenz \mid (1+S)=0$. We can now record the following consequence of Theorem  \ref{thm:solfromwp}.

\begin{corollary}\label{cor:faytodsh} Let $\Benz_1\bi{X}{Y} \in A[[X,Y]]$ be such that $\Benz_1\bi{-X}{-Y}=-\Benz_1\bi{X}{Y}$ and such that
	\begin{align*}
	\Fenz\bi{X}{Y} := -\frac{1}{2}\left(\frac{1}{X} + \frac{1}{Y} \right) + \Benz_1\bi{X}{Y} 
	\end{align*} 
	satisfies the Fay identity \eqref{eq:fay}. 
	Then 
	\[  \Benz_2 :=\frac{1}{3} \penz_\Benz {\mid\! (1+T^{-1})}-\frac{1}{12} \renz_{\Benz_1}^\ast {\mid\! (5 -3 U + U \epsilon)} -\frac{1}{12}\renz_{\Benz_1}^\shuffle {\mid\! (T^{-1} ( 5 - 3 \epsilon + U))}\] with $\penz_\Benz\bi{X_1,X_2}{Y_1,Y_2}:=\Benz_1\bi{X_1}{Y_1}\Benz_1\bi{X_2}{Y_2}$  gives a solution to \eqref{eq:actiongenrelation}, i.e. 
\begin{align*}
		\penz_\Benz &=  \Benz_2 {\mid\! (1+\epsilon)} + \renz_{\Benz_1}^\ast \\
		&=    \Benz_2  {\mid T(1+\epsilon)} + \renz_{\Benz_1}^\shuffle \,.
\end{align*}
We obtain in particular a realization of $\de_k$ with values in $A$, for every $k\geq 1$.
\end{corollary}
\begin{proof}
	We already know that, if $\Fenz$ satisfies the Fay identity, then the product
	\begin{align*}
	\Fenz\bi{X_1}{Y_1}&\Fenz\bi{X_2}{Y_2} =\frac{1}{4} \left(\frac{1}{X_1} + \frac{1}{Y_1} \right)\left(\frac{1}{X_2} + \frac{1}{Y_2} \right)\\
	&-\frac{1}{2}\left( \left(\frac{1}{X_2} + \frac{1}{Y_2} \right) \Benz_1\bi{X_1}{Y_1} +  \left(\frac{1}{X_1} + \frac{1}{Y_1} \right) \Benz_1\bi{X_2}{Y_2}  \right) + \Benz_1\bi{X_1}{Y_1}\Benz_1\bi{X_2}{Y_2}
	\end{align*}
	is contained in $\mathcal{W}^+\!(A)$. However, since $\left(\frac{1}{X_1} + \frac{1}{Y_1} \right)\left(\frac{1}{X_2} + \frac{1}{Y_2} \right) \in \mathcal{W}^+\!(A)$, the claim follows from Theorem \ref{thm:solfromwp} with 
	\[ \widetilde{\penz}_\Benz = \Fenz\bi{X_1}{Y_1}\Fenz\bi{X_2}{Y_2}- \frac{1}{4} \left(\frac{1}{X_1} + \frac{1}{Y_1} \right)\left(\frac{1}{X_2} + \frac{1}{Y_2} \right)\,.\]
\end{proof}

If we take $\Benz_1\bi{X}{Y} = \sum_{\substack{k\geq 1\\d\geq 0}} b\bi{k}{d} X^{k-1} \frac{Y^{d}}{d!}$ in Corollary \ref{cor:faytodsh}, then we obtain a realization of $\de_k$ with $G\bi{k}{d} \mapsto b\bi{k}{d}$ and $P\bi{k_1,k_2}{d_1,d_2} \mapsto  b\bi{k_1}{d_1} b\bi{k_2}{d_2}$. We are therefore interested in describing all solutions to the Fay identity. This question was answered by the third author in \cite{Matthes} (see also \cite{Polishchuk}\footnote{Note that in \emph{loc.cit.} the Fay identity is referred to as the ``associative Yang-Baxter equation''.}) by showing that all such solutions can be obtained from the \emph{Kronecker function}, which is defined by 
\begin{align}\label{eq:defhoneckerfunction}\begin{split} 
	\Kenz_q\bi{X}{Y} &=-\frac{1}{2}    \sum_{m=0}^\infty\frac{e^{-X-mY}q^m}{1-q^m e^{-X}} +   \frac{1}{2}\sum_{m=0}^\infty\frac{e^{Y+mX}q^m}{1-q^m e^{Y}}\\
	&= -\frac{1}{2}\left(\frac{1}{X} + \frac{1}{Y} \right) + \sum_{\substack{r,s\geq 0\\r+s \text{ odd}}} \frac{|r-s|!}{r!}  \left( q\frac{d}{dq} \right)^{\min\{r,s\}} G_{|r-s|+1}(q)   \,X^r  \frac{Y^s}{s!}\,,
\end{split}
\end{align}
where the \emph{Eisenstein series} of weight $k\geq 1$ is given by $G_k(q):=0$, if $k$ is odd, and  for even $k\geq 2$  by 
\begin{align}\label{eq:defeisensteinseries}
	G_k(q) := -\frac{B_k}{2k!}+ \frac{1}{(k-1)!}\sum_{n=1}^\infty \frac{n^{k-1}q^n}{1-q^n}\,.
\end{align}
\begin{lemma} \label{lem:krosatfay}The Kronecker function $\Kenz_q$ satisfies the Fay identity \eqref{eq:fay}.
\end{lemma}
\begin{proof} This follows from a classical theta function identity due to Riemann. For details, see \cite[Proposition 3.4]{Matthes} or \cite[Proposition 5]{Z}. 
\end{proof}

\begin{remark}\label{rem:proveofhoneckersatisfiesfay}
\begin{itemize}
\item[(i)]
Another way of proving Lemma \ref{lem:krosatfay} is by observing that the converse of Corollary \ref{cor:faytodsh} is also true, i.e., that every solution of \eqref{eq:qdsh} with $\penz_\Benz\bi{X_1,X_2}{Y_1,Y_2}:=\Benz_1\bi{X_1}{Y_1}\Benz_1\bi{X_2}{Y_2}$ gives, after removing the odd weight parts, an element $\Fenz$ satisfying the Fay identity.
\item[(ii)]
In \cite{B2}, another realization of $\de_k$ for all $k\geq 1$ is given by using the algebraic structure of $q$-analogues of multiple zeta values which was first introduced in \cite{B1}. This realization coincides with the realization constructed here in depth one for even weights. Therefore, the Fay identity for the Kronecker function may also be proved, purely combinatorially, using the results of \cite{B2}.
\end{itemize}
\end{remark}
Now denote by $\widetilde{\mathcal{M}}:= \Q[G_2,G_4,G_6]$ the space of \emph{quasimodular forms} with rational coefficients, \cite{KZ}. It is graded for the weight, $\widetilde{\mathcal{M}}=\bigoplus_{k\geq 0}\widetilde{\mathcal{M}}_k$
\begin{proposition}\label{prop:kroneckerprop} For every $K\geq 1$, the Kronecker function gives rise to a realization $\rho^\Kenz_K: \de_K \twoheadrightarrow \widetilde{\mathcal{M}}_K$, the \emph{Kronecker realization}, which is given by 
\begin{align*}
		\rho^\Kenz_K: G\bi{k}{d}&\longmapsto \frac{(k-d-1)!}{(k-1)!} \left(q \frac{d}{dq}\right)^d G_{k-d},\qquad (k>d\geq 0, k+d=K \text{ even})\,,
\end{align*}
and for $k_1,k_2 \geq 1$, with $k_1+k_2=K$ even, by 
\begin{align*}\begin{split}		
\rho^\Kenz_K:	P\bi{k_1,k_2}{0,0}&\longmapsto \,G_{k_1} G_{k_2}\,, \\
\rho^\Kenz_K:	G\bi{k_1,k_2}{0,0}&\longmapsto	\,\frac{1}{3} G_{k_1} G_{k_2} + \frac{(-1)^{k_1}}{3} \sum_{\substack{l_1 + l_2 = k_1+k_2\\l_1,l_2 \geq 2 \text{ even}}}  \binom{l_2-1}{k_1-1} G_{l_1} G_{l_2}   \\
	&-\frac{1}{12}\left(5 + 3(-1)^{k_1} \binom{k_1+k_2-1}{k_1-1}-(-1)^{k_1}\binom{k_1+k_2-1}{k_1} \right) G_{k_1+k_2}\\
	&- \frac{5 \delta_{k_2,1} G'_{k_1-1}}{12(k_1-1)}   +\frac{\delta_{k_1,1} G'_{k_2-1}}{4 (k_2-1)}   + \frac{(-1)^{k_2}}{12 (k_1+k_2-2)} \binom{k_1+k_2-2}{k_1-1}  G'_{k_1+k_2-2}\,, 
\end{split}
\end{align*}
where $G'_k = q \frac{d}{dq} G_k$, for $k\geq 2$, and $G'_1:=G_2$.
\end{proposition}
\begin{proof}
This follows by using $\Kenz_q$ in  Corollary \ref{cor:faytodsh} together with Lemma \ref{lem:krosatfay} and then considering the coefficients of the resulting power series. That the image of this realization is precisely the space $\widetilde{\mathcal{M}}_k$ follows from the well-known fact, essentially due to Rankin, that every quasimodular form can be written as a sum of products of derivatives of Eisenstein series and $G_2$ (see, for example, \cite[p.20]{GKZ}).
\end{proof}
In particular, we can also consider just the constant terms in the Kronecker function and obtain a realization of $\rho^{\text{Ber}}_k: \de_k \rightarrow \Q$. In depth one, this realization satisfies $G\bi{k}{0} \mapsto -\frac{B_k}{2k!}$, for $k>1$. Together with Proposition \ref{prop:dktoek}, this gives a realization $\rho^{\text{Ber}}_k \circ \sigma_k$ of $\dz_k$ with values in $\Q$, which is called the \emph{Bernoulli realization} in \cite{GKZ}. Details of this realization can be found in \cite{B2}, where another realization of $\de_k$, is introduced. The latter is given by so-called combinatorial double Eisenstein series, and will be studied in more detail in \cite{BB}.
\begin{proposition} \label{prop:derivcommute}For every $k\geq 1$, the following diagram commutes
\begin{center}
\begin{tikzcd}
  \de_k \arrow[r, "\rho^\Kenz_k"] \arrow[d,"\partial_k",swap]
    & \widetilde{\mathcal{M}}_k \arrow[d, "q\frac{d}{dq}"] \\
  \de_{k+2} \arrow[r, "\rho^\Kenz_{k+2}",swap]
&  \widetilde{\mathcal{M}}_{k+2} 
\end{tikzcd}
\end{center}
where $\partial_k$ is the map given in Proposition \ref{prop:partialk}.
\end{proposition}
\begin{proof}
Denote by $\Benz_1$, $\Benz_2$ and $\penz_\Benz$ the power series in Corollary \ref{cor:faytodsh} in the case $\Fenz = \Kenz_q$. Thus, $\Benz_1$ is given by the non-polar part in \eqref{eq:defhoneckerfunction}, and it is easy to see that
\begin{align}\label{eq:qdqb1}
    q\frac{d}{dq} \Benz_1\bi{X}{Y} = \frac{d}{dX} \frac{d}{dY}\Benz_1\bi{X}{Y}\,.
\end{align}
In light of the proof of Proposition \ref{prop:partialk}, it remains to show that
\begin{align}\label{eq:qdqb2}
    q\frac{d}{dq}  \Benz_2\bi{X_1,X_2}{Y_1,Y_2} = \left(\frac{d}{dX_1}\frac{d}{dY_1} + \frac{d}{dX_2}\frac{d}{dY_2}   \right) \Benz_2\bi{X_1,X_2}{Y_1,Y_2}\,.
\end{align}
For this, we first observe that \eqref{eq:qdqb1} implies for $\bullet \in \{\ast, \shuffle\}$
\begin{align*}
     \left(\frac{d}{dX_1}\frac{d}{dY_1} + \frac{d}{dX_2}\frac{d}{dY_2}   \right) \renz^\bullet_{\Benz_1}\bi{X_1,X_2}{Y_1,Y_2}  =  q\frac{d}{dq}\renz^\bullet_{\Benz_1}\bi{X_1,X_2}{Y_1,Y_2} \,.
\end{align*}
Moreover, using the definition of the action \eqref{eq:defmatthesaction}, we deduce that
\begin{align*}
   \left(\frac{d}{dX_1}\frac{d}{dY_1} + \frac{d}{dX_2}\frac{d}{dY_2}   \right) (\renz^\bullet_{\Benz_1}\!\!\mid\!\!\gamma)\!\bi{X_1,X_2}{Y_1,Y_2}  = \det(\gamma)^2 q \frac{d}{dq} \renz^\bullet_{\Benz_1}\bi{X_1,X_2}{Y_1,Y_2}\,,
\end{align*}
for every $\gamma \in \operatorname{GL}_2(\mathbb Z)$.
A similar argument also works for $\penz_\Benz$ and since all matrices in our case are in $\GL_2(\Z)$ we obtain \eqref{eq:qdqb2} by the definition of $\Benz_2$ in Corollary \ref{cor:faytodsh} as a linear combination of $\renz^\bullet_{\Benz_1}\!\!\mid\!\!\gamma$ and $\penz_\Benz\!\!\mid\!\!\gamma$ for some elements $\gamma \in \Z[\GL_2(\Z)]$.
\end{proof}

\section{Families of relations in \texorpdfstring{$\de_k$}{}}
In this section, we describe some consequences of the defining relation \eqref{eq:derel} of the formal double Eisenstein space. Applying one of the various realizations discussed before then gives relations among either Eisenstein series and their derivatives, Bernoulli numbers, or zeta values.
Firstly we have the following analogue of the sum formula, \cite[Theorem 1]{GKZ}. 

\begin{proposition}(Sum formula) For $k\geq 2$ and $d\geq 0$, we have
	\begin{align*}
	\sum_{\substack{k_1+k_2=k\\d_1+d_2=d\\(k_1,d_2)\neq (1,0) }} (-1)^{d_2} \,\binom{d}{d_2} G\bi{k_1,k_2}{d_1,d_2} = G\bi{k}{d} -\frac{1}{d+1} G\bi{k-1}{d+1}\,,
	\end{align*}
	where the sum runs over all $k_1,k_2\geq 1$ and all $d_1,d_2\geq 0$.
\end{proposition}
\begin{proof} This follows by considering the coefficient of $X^{k-2} Y^d$ in \eqref{eq:qdsh}  after setting $X_1=X, X_2=0, Y_1=0$ and $Y_2=Y$.
\end{proof}
Similarly to \cite[Theorem 1]{GKZ}, one can also give odd and even variants of the above sum formula, which we will omit here. A more interesting family of relations can be obtained from the following lemma, which generalizes \cite[Lemma 4.5]{B3}, and which we state again by using the group actions on the generating series $\penz$, $\Genz_1$ and $\Genz_2$ of the elements in $\de_k$.
\begin{lemma} \label{lem:dzparitylemma} For $A= \epsilon U\epsilon $ we have 
	\begin{align*}
	\Genz_2 \mid (1 - \sigma)  &= \penz \mid\! (1-\delta)(1 + A - SA^2) - (\renz^\ast_{\Genz_1} + \renz^\shuffle_{\Genz_1} \mid\! T^{-1}\epsilon) \mid\! (1+A+A^2) \,.
	\end{align*}
\end{lemma}
\begin{proof}
First, notice that $A  = \epsilon U\epsilon = T \epsilon T^{-1} \epsilon = \pmatr{0 & 1\\ -1 & 1}$ and that we have $A^3 = \sigma$. Now recall that by \eqref{eq:qdsh}
\begin{align*}
		\penz &=  \Genz_2 {\mid\! (1+\epsilon)} + \renz_{\Genz_1}^\ast =    \Genz_2 {\mid T(1+\epsilon)} + \renz_{\Genz_1}^\shuffle \,,
\end{align*}
from which we obtain
\begin{align*}
\Genz_2 \mid\! \epsilon &= - \Genz_2 +  \penz -  \renz_{\Genz_1}^\ast\,,\\
\Genz_2  \mid\! T \epsilon T^{-1} &=  -   \Genz_2  + \penz \mid\! T^{-1} - \renz_{\Genz_1}^\shuffle   \mid\! T^{-1}   \,.
\end{align*}
This implies 
\begin{align*}
\Genz_2 \mid\! A &= \Genz_2 \mid\!  (T \epsilon T^{-1}) \epsilon = \left( -   \Genz_2  + \penz \mid\! T^{-1} - \renz_{\Genz_1}^\shuffle   \mid\! T^{-1} \right) \mid\! \epsilon \\
&= \Genz_2 + \underbrace{\penz \mid\! (T^{-1}\epsilon - 1) + \renz_{\Genz_1}^\ast - \renz^\shuffle_{\Genz_1} \mid(T^{-1}\epsilon) }_{=: \mathfrak{K}}\,.    
\end{align*}
Iterating this identity two more times yields
\begin{align*}
\Genz_2 \mid\! A^3 &= \Genz_2 + \mathfrak{K} \mid\! (1+A+A^2)\,.
\end{align*}
By direct calculation, one can check that the action of  $(T^{-1}\epsilon - 1)(1+A+A^2)$ and $-(1-\delta)(1 + A - SA^2)$ is the same on $\penz$, since $\penz\mid \epsilon = \penz$.
\end{proof}

\begin{theorem}\label{thm:parity}(Parity) For $K\geq 3$ odd, every $G\bi{k_1,k_2}{d_1,d_2}$ with $k_1+k_2+d_1+d_2=K$ can be written as a linear combination of $G\bi{-}{-}$ and $P\bi{\,-\,,\,-\,}{\,-\,,\,-\,}$.
\end{theorem}
\begin{proof}
Due to the action of $(1-\sigma)$ in Lemma \ref{lem:dzparitylemma}, the even weight terms on the left-hand side vanish, and the odd weight terms just get multiplied by $2$. This gives an explicit representation of the coefficients in $\Genz_2$ as a linear combination of the coefficients in $\Genz_1$ and $\penz$.
\end{proof}

Applying $\pi_k$ and then the Euler realization mentioned in the introduction, Theorem \ref{thm:parity} implies in particular Euler's result that every double zeta value of odd weight can be written as a linear combination of products of single zeta values.
An explicit formula which gives the special case of  Theorem \ref{thm:parity} can be found in \cite[Theorem 4.6]{B3}, where the corresponding statement for $\dz_k$ is given. Applying $\pi_k$ then gives the $d_1=d_2=0$ case of Theorem \ref{thm:parity}. 

The next theorem gives the even weight case of Lemma \ref{lem:dzparitylemma}. Again we merely state a special case by giving the formula coming from the coefficients of $X_1^{k_1-1} X_2^{k_2-1}$, which can also be obtained by applying $\sigma_k$ to \cite[Theorem 4.9]{B3}. It would be possible to obtain a more general statement from Lemma \ref{lem:dzparitylemma} by working out all coefficients of $X_1^{k_1-1} X_2^{k_2-1} Y_1^{d_1} Y_2^{d_2}$ for $k_1+k_2+d_1+d_2$ even. 
\begin{theorem}\label{thm:relprodandg}
For all $k_1,k_2\geq 1$, with $k=k_1+k_2\geq 4$ even, we have 
{\small
\begin{align} \label{eq:relprodandg}\begin{split}
		\frac{1}{2}\left(  \binom{k_1+k_2}{k_2} - (-1)^{k_1}\right) G\bi{k}{0} =  &\sum_{\substack{j=2\\j \text{even}}}^{k-2} \left( \binom{k-j-1}{k_1-1} + \binom{k-j-1}{k_2-1} - \delta_{j,k_1} \right)  P\bi{j,k-j}{0,0}   \\
		& + \frac{1}{2} \left( \binom{k-3}{k_1-1} + \binom{k-3}{k_2-1}  + \delta_{k_1,1} + \delta_{k_2,1} \right) G\bi{k-1}{1}\,. 
		\end{split}
\end{align}}
\end{theorem}
\begin{proof}
This follows by comparing coefficients of $X_1^{k_1-1} X_2^{k_2-1}$ on both sides of Lemma \ref{lem:dzparitylemma}. 
\end{proof}
We record the following special cases of Theorem \ref{thm:relprodandg}.
\begin{corollary} \label{cor:mfprod}
	\begin{enumerate}[i)]
		\item For even $k\geq 4$ we have
		\begin{align*}
	 G\bi{k-1}{1}	 =   \frac{k+1}{2} G\bi{k}{0} + \sum_{\substack{k_1+k_2=k \\ k_1, k_2\geq 2 \text{ even} }}  P\bi{k_1,k_2}{0,0} \,.
		\end{align*}
		\item 	 For all even $k\geq 6$ we have
		\begin{align*}
			\frac{(k+1)(k-1)(k-6)}{12}	 G\bi{k}{0} =  \sum_{\substack{k_1+k_2 = k\\k_1,k_2\geq 4 \text{ even}}} (k_1-1)(k_2-1)   \,P\bi{k_1,k_2}{0,0}\,.
		\end{align*}
	\end{enumerate}
\end{corollary}
\begin{proof}
Part i) is the $k_1=1$ case of Theorem \ref{thm:relprodandg}. Part ii) follows by considering $k\!-\!3$-times the $(k_1,k_2)=(k-2,2)$ case and then subtracting $2$-times the $(k_1,k_2)=(k-3,3)$ case.
\end{proof}

\begin{example}
    \begin{enumerate}[i)]
    \item We have 
        \begin{align*}
        	G\bi{8}{0}  = \frac{6}{7}  P\bi{4,4}{0,0},\,\quad  G\bi{10}{0}  = \frac{10}{11} P\bi{4,6}{0,0} \,.
        \end{align*}
        \item Analogues of the Ramanujan differential equations are satisfied 
		\begin{align*}
			2 G\bi{3}{1} &= 5 G\bi{4}{0} - 2 P\bi{2,2}{0,0}\,,\qquad 
			4 G\bi{5}{1} = 8 G\bi{6}{0} - 14 P\bi{2,4}{0,0}\,,\\
			6 G\bi{7}{1} &= \frac{120}{7} P\bi{4,4}{0,0} - 12 P\bi{2,6}{0,0}\,.
		\end{align*}
    \end{enumerate}
\end{example}

Applying the Kronecker realization in Proposition \ref{prop:kroneckerprop} we get the well-known relations $G_8 = \frac{6}{7} G_4^2$, $G_{10} = \frac{10}{11} G_4  G_6$ as well as the Ramanujan differential equations
\begin{align*}
	q \frac{d}{dq} G_{2} &= 5 G_{4} - 2 G_{2}^{2}\,,\\
	q \frac{d}{dq} G_{4} &= 8 G_{6} - 14 G_{2}G_{4}\,,\\
	q \frac{d}{dq} G_{6} &= \frac{120}{7} G_{4}^2 - 12 G_{2} G_{6}\,.
\end{align*}
Using the algebraic approach to Lemma \ref{lem:krosatfay}, as explained in Remark \ref{rem:proveofhoneckersatisfiesfay}, we can derive the above identities in a purely combinatorial manner, in the spirit of \cite{S} or \cite{HST}.

\end{document}